\documentclass[11pt, a4paper]{article}

\usepackage{amsmath, amsthm, amsfonts, amssymb}

\usepackage{setspace}
\usepackage{fullpage}
\usepackage{enumitem}

\usepackage{floatpag}
\usepackage[dvipsnames]{xcolor}

\usepackage[initials]{amsrefs} 

\usepackage{float}
\usepackage{caption}
\usepackage{subcaption} 

\usepackage{hyperref}

\hypersetup{colorlinks=true,
	citecolor=blue,
	filecolor=blue,
	linkcolor=blue,
	urlcolor=blue
}
\usepackage{cleveref}

\usepackage{comment}

\newtheorem{lemma}{Lemma}
\newtheorem{theorem}{Theorem}

\newcommand{\cG}{\mathcal{G}}
\newcommand{\cH}{\mathcal{H}}

\newcommand{\ex}{{\rm  ex}}
\newcommand{\Prob}{{\rm  Prob}}

\begin{document}
	\title{Extremal numbers for cycles in a hypercube}
		\date{\vspace{-5ex}}
	\author{
		Maria Axenovich
		\thanks{
			Karlsruhe Institute of Technology, Karlsruhe, Germany;
			email:
			\mbox{\texttt{maria.aksenovich@kit.edu}}. 
}
		}     
			
	\maketitle

	\begin{abstract}
	Let $\ex(Q_n, H)$ be the largest number of edges in a subgraph $G$ of a hypercube $Q_n$ such that there is no subgraph of $G$ isomorphic to $H$. We show that for any integer $k\geq 3$ 
	$$\ex(Q_n, C_{4k+2})= O(n^{\frac{5}{6} + \frac{1}{3(2k-2)}} 2^n).$$ 
\end{abstract}

	A {\it hypercube} $Q_n$, where $n$ is a natural number, is a graph on  a vertex set $\{A: A\subseteq [n]\}$ and an edge set consisting of all pairs $\{A,B\}$, where $A\subseteq B$ and $|A|=|B|-1$. Here, $[n]= \{1, \ldots, n\}$.
For a graph $H$, let 	$\ex(Q_n, H)$ be the largest number of edges in a subgraph $G$ of a hypercube $Q_n$ such that there is no subgraph of $G$ isomorphic to $H$.  The behaviour of the function $\ex(Q_n, H)$ is not well understood in general and it is not even known for what graphs $H$, $\ex(Q_n, H) = \Omega(||Q_n||)$, where $||Q_n||=n2^{n-1}$ denotes the number of edges of $Q_n$.  See the following papers on the topic: \cite{baber, balogh, AKS, ARSV, AM,TW, O}. Conlon \cite{C} observed a connection between extremal numbers in the hypercube and classical extremal numbers for uniform hypergraphs. That allowed to determine a large class of graphs $H$ such that $\ex(Q_n, H) = o(||Q_n||)$, i.e., the class of graphs with zero Tur\'an density.\\

In this note we focus on the case when $H$ is an even cycle. Here, a $2\ell$-cycle, denoted $C_{2\ell}$, is a graph on $2\ell$ vertices and $2\ell$ edges such that the edges are pairs of consecutive vertices with respect to some circular ordering of the vertices. It is known that $\ex(Q_n, C_4) = \Omega(||Q_n||)$ and $\ex(Q_n, C_6) = \Omega(||Q_n||)$, see Chung~\cite{chung}, Conder~\cite{conder}, and Brass {et al.}~\cite{brass}. Chung \cite{chung} showed that $\ex(Q_n, C_{4k})=o(||Q_n||)$, for any integer $k\geq 2$.
F\"uredi and \"Ozkahya~\cite{FO, FO2}   extended Chung's results by showing that  $\ex(Q_n, C_{4k+2})=o(||Q_n||)$, for any integer $k\geq 3$. Thus if $\ell\not\in\{2, 3, 5\}$, $\ex(Q_n, C_{2\ell}) = o(||Q_n||)$. While 
$\ex(Q_n, C_{2\ell}) =\Omega(||Q_n||)$ for $\ell\in \{2, 3\}$, it remains unclear whether $\ex(Q_n, C_{10}) =\Omega(||Q_n||)$. \\

Conlon \cite{C} noted that one can obtain a non-trivial lower bound on $\ex(Q_n, C_{2\ell})$ using Lov\'asz Local Lemma.  For $\ell \geq 2$, $\ex(Q_n, C_{2\ell}) = \Omega(n^{\frac{1}{2}  + \frac{1}{4\ell -2}}2^n)$, as  proven using deletion method, see Conlon' homepage  \cite{C1}. In Appendix, we include the proof of this lower bound using Lov\'asz  Local Lemma.\\

Considering more specific upper bounds for cycles with zero Tur\'an density,  Conlon \cite{C}  proved for $k \geq 2$ that
$\ex(Q_n, C_{4k}) \leq c_k  n^{-\frac{1}{2} + \frac{1}{2k}}  ||Q_n||.$  F\"uredi and \"Ozkahya~\cite{FO}, \cite{FO2}  showed that  $\ex(Q_n, C_{4k+2}) = O(n^{-q_k} ||Q_n|| ),$ where $q_k = \frac{1}{2k+1}$ for $k\in \{3,5, 7\}$, and $q_k = \frac{1}{16} - \frac{1}{16(k-1)}$ for any other integer  $k\geq 3$. Note that for large $k$, this upper bound on $\ex(Q_n, C_{4k+2})$ behaves as  $n^{\frac{15}{16}+ \epsilon}2^n$, for small positive $\epsilon$.  Here, we improve the leading term in the exponent from $\frac{15}{16}$ to $\frac{5}{6}$:

\begin{theorem}\label{main}
Let $\ell$ be an odd integer, $\ell \geq 7$. Then $\ex(Q_n, C_{2\ell})= O\left(n^{\frac{5}{6} + \frac{1}{3(\ell-3)}} 2^n\right)$.
\end{theorem}

The main ingredient of the proof is also due to Conlon \cite{C} using partite representation. Our contribution here is finding a suitable representation and bounding the respective extremal hypergraph function.\\

We need some definitions and notations.  For a set $V$ and a non-negative integer $r$, we denote by $\binom{V}{r}$ the set of all $r$-element subsets of $V$.  An {\it $r$-graph} on a vertex set $V$ is a pair $(V, E)$, where $E\subseteq \binom{V}{r}$ is the set of edges.  Note that a graph is a $2$-graph. We denote the number of edges in an $r$-graph $H$ by  $||H||$. If $H$ is an $r$-graph, then $\ex(n, H)$ denotes the largest number of edges in an $n$-vertex $r$-graph not containing $H$ as a subgraph.
An $r$-graph is {\it $k$-partite} if its vertex set could be partitioned into  $k$  parts such that any edge  contains at most one vertex from each part. 
Edges $\{x,y\}$ and $\{x,y,z\}$ are written  simply as $xy$ and $xyz$, respectively. For a $3$-graph $\cH$ and its vertex $x$, the {\it link graph } $L(x)$ is the graph with no isolated vertices whose edge set  consists of all pairs $yz$ such that $xyz$ is an edge of $\cH$. Let $H$ be a bipartite graph. Let the {\it two-lift} of $H$ be the $3$-graph  $\cH$ with vertex set $\{a, b\}\cup V(H)$, where $a\not\in V(H)$ and $b\not\in V(H)$,  $a\neq b$,  each edge of $\cH$ contains exactly one of  $a$ or $b$,  and $L(a)=L(b)=H$.\\

For a non-negative $k$, the $k$th layer of $Q_n$ is the set $\binom{[n]}{k}$ of vertices. We say that a subgraph $H$ of $Q_n$ has  a {\it $k$-partite representation}  $\cH$  if $H$ is isomorphic to a graph $H'$ with a vertex set contained in  $\binom{[n]}{k-1}\cup \binom{[n]}{k}$ such that   $V(H') \cap \binom{[n]}{k}$ is an edge-set of a $k$-partite $k$-graph.  In this case, we denote such a hypergraph with no isolated vertices by $\cH(H)$ and call it a  $k$-partite {\it representation} of $H$.  For example, if $H$ is an $8$-cycle, it has a $2$-partite representation with edges 
$12, 23, 34, 14$ corresponding to an $8$-cycle  with vertices $1, 12, 2, 23, 3, 34, 4, 14, 1$, in order. A more general result by Conlon implies the following in the case when $k=3$.

\begin{theorem} [Conlon \cite{C}]\label{Conlon-rep-extremal}
If $H$ is a subgraph of $Q_n$ that has a  $3$-partite representation $\cH$ such that 
$\ex(n, \cH) \leq \alpha n^3$, then $\ex(Q_n, H) \leq \alpha^{1/3}2^nn$.
\end{theorem}

\begin{lemma}\label{embedding-cycle}
A cycle $C_{2\ell}$, where $\ell \geq 7$ is an odd integer, has a $3$-partite representation $\cH_\ell$ such that $\cH_\ell $ is a subgraph of a two-lift of a cycle of length $\ell-3$ with  a pendant edge.
\end{lemma}

\begin{proof}

We shall construct a cycle $C=C_{2\ell}$ for odd $\ell=2k+1$ with vertices in the third and second layer of $Q_n$. We shall define $A_\ell$  to be the ordered set of vertices of $C$ in the third layer, $B_C$ to  be a respective set of vertices in the second layer, and let  $C$ be  built by taking vertices from $A_\ell$ and $B_\ell$ alternatingly in order.
Let 
\begin{eqnarray*}
A_{2k+1} & = & (ax_1y_1,    ax_2y_1,  ax_2y_2,  \ldots, ax_{k-1}y_{k-2},  ax_{k-1}y_{k-1}, ~~ bx_{k-1}y_{k-1},   bx_{k-1}y_{0},  bx_1y_0,  bx_1y_1),\\
B_{2k+1} & = & (ay_1,    ax_2,  ay_2,  \ldots, ax_{k-1}, ~~  x_{k-1}y_{k-1},   bx_{k-1},  by_0,   bx_1,  x_1y_1).
\end{eqnarray*}

Then $C$ has a $3$-partite representation $\cH_\ell$ with the  edge set  equal to the  underlying unordered set of $A_\ell$.
Indeed, we see that  $H_\ell$ is a three-partite $3$-graph with parts $\{a,b\}$, $\{x_1, x_2, \ldots, x_{k-1}\}$, and $\{y_0, y_1, y_2, \ldots, y_{k-1}\}$. 
One can visualise $H_\ell$ by considering link graphs $L(a)$ and $L(b)$ - these graphs are paths of lengths $2k-3$ and $4$ respectively,   sharing only vertices $x_1, y_1, x_{k-1}$, and $y_{k-1}$ and edges $x_1y_1$ and $x_{k-1}y_{k-1}$.  The union of these two paths is a cycle of length $2(k-1)= \ell - 3$ and a pendant edge.
\end{proof}

\begin{lemma}\label{3-vs-2-graph}
Let $H$ be a bipartite graph and $\cH$ be a two-lift of $H$. If $\ex(n, H) =O(n^\gamma)$, then $\ex(n, \cH) = O(n^{\frac{\gamma+4}{2}})$.
\end{lemma}

\begin{proof}
Let $\gamma, c' >0$, $n$ sufficiently large, and $\ex(n, H) < c'n^\gamma$.   Let $\cG$ be a $3$-graph with vertex set $V$, $|V|=n$ and $cn^\sigma$ edges, where $\sigma = \frac{\gamma+4}{2}$ and $c$ is a constant, $c'=c^2$.
We shall show that $\cG$ contains a two-lift $\cH$ of $H$.

Let $U= \binom{V}{2}$ and $G$ be a bipartite graph with partite sets $V$ and $U$ and 
edge set $\{ \{x, zz'\}:  x\in V, zz'\in U, xzz'\in E(\cG)\}$. 
Then $||G|| = 3||\cG||=3cn^{\sigma}$. 
Let $q$ be the largest integer such that $G$ contains a complete bipartite subgraph with two vertices in $V$ and $q$ vertices in $U$. 
We shall show that  $q\geq  n^{2\sigma-4}$.

To do this, we shall apply a standard double counting and reduction to smaller uniformity, already used by Erd\H{o}s \cite{E64} in bounding the extremal function for complete $k$-partite $k$-graphs.  Let $X$ be the number of  stars in $G$  with $2$ edges and center in $U$.  Then $X\geq \sum_{u\in U} \binom{deg_G(u)}{2} \geq |U| \binom{||G||/|U|}{2} \geq c^2 n^2 n^{2(\sigma-2)}$. 
On the other hand, since every pair of vertices from $V$ belongs to at most $q$ such stars,  $X \leq q \binom{n}{2} \leq qn^2$. Thus $q \geq c^2  n^{2\sigma - 4} = c' n^\gamma$. 
Let $Q\subseteq U$, $|Q|=q$, $a, b\in V$, such that $a,b,$ and $Q$ induce $K_{2, q}$ in $G$. Then 
$a,b$ and $Q$ correspond to a two-lift $\cH'$ of a graph $H'$ with the edge-set $Q$. 
Since $|Q|=q = c'n^\gamma >\ex(n, H)$, we see that $H$ is a subgraph of $H'$. Thus $\cG$ contains a two-lift of $H$.
\end{proof}

\begin{proof}[Proof of Theorem \ref{main}]
It was shown by Erd\H{o}s \cite{E} as well as by Bondy and Simonovits \cite{BS},  that $\ex(n, C_{2m}) = O(n^{1 + \frac{1}{m}})$.
This implies that $\ex(n, H) \leq O(n^{1+\frac{2}{\ell-3}})$, where $H$ is a cycle of length $\ell-3$ with a pendant edge, for odd $\ell\geq 7$. Now, using Theorem \ref{Conlon-rep-extremal} and Lemmas \ref{embedding-cycle} and \ref{3-vs-2-graph}, 
 with $\gamma = 1+\frac{2}{\ell-3}$, $\frac{\gamma+4}{2} =  \frac{5 + \frac{2}{\ell-3}}{2}$, and $\alpha =-\frac{1}{2} + \frac{1}{\ell-3}$, we have $\ex(Q_n, C_{2\ell}) \leq O(\alpha^{1/3}n2^n)= O(n^{1-\frac{1}{6} + \frac{1}{3(\ell-3)}} 2^n)$.
\end{proof}

{\bf Remark} ~~ After uploading this note, the author learned about a better bound obtained by Istvan Tomon  \cite{T},  page 28.,  as a corollary of a much more general result on topologically defined hypergraphs:
$\ex(Q_n, C_{2\ell})= O(n^{\frac{2}{3}+ \delta} 2^n), $ for some $\delta = O(\frac{\log \ell}{\ell})$. \\

{\bf Acknowledgements} ~~The author thanks David Conlon for bringing her attention to \cite{C1}, Frithjof Marquardt for discussions, and Zixiang Xu for pointing out \cite{T}. The research was supported in part by the DFG grant FKZ AX 93/2-1.

\section{Appendix}
\begin{theorem}
For any integer $\ell$, $\ell \geq 2$, $ex(Q_n, C_{2\ell}) = \Omega(n^{\frac{1}{2}  + \frac{1}{4\ell -2}}2^n)$.
\end{theorem}

\begin{proof}
Consider a  random edge-coloring of   $Q_n$, where the colors for the edges  are chosen i.i.d.  with probability $p= cn^{-a}$ for each of the $1/p$ colors, where the constants $c$ and $a$ will be determined later.
For a copy $C$ or $C_{2\ell}$, let $A_C$ be the probability of a bad event that all edges of $C$ have the same color. 
Denoting $P=\Prob (A_C)$, we have $P \leq p^{2\ell-1}$. Note that $A_C$ is mutually  independent of  all events $A_{C'}$ except for those $C'$s  that share an edge with $C$.\\

To determine the maximum degree in the dependency graph, consider first $N(Q_n, C_{2\ell})$, the number of $2\ell$-cycles in $Q_n$.  Represent  the vertices of $Q_n$ as binary vectors, consider a $2\ell$-cycle $C$ and  the set  $S_C$ of all positions at which  some two adjacent in $C$ vertices differ. Then  $|S_C|\leq \ell$. Since there are at most $n^\ell$ ways to choose the set $S_C$ and at most $2^n$ ways to fix the remaining positions, in which all vertices of $C$ coincide, we have 
that $N(Q_n, C_{2\ell})\leq c_\ell n^{\ell} 2^n$.

Let $x$ be the number of $2\ell$-cycles containing a given edge. By symmetry of $Q_n$ we see that $x$ is independent of the choice of the edge.  Let $X$ be the set of pairs $(C, e)$, where $C$ is a cycle of length $2\ell$ in $Q_n$ and $e$ is an edge of $C$. Then $|X| = N(Q_n, C_{2\ell})\cdot 2\ell$ and $|X|= n2^{n-1} \cdot x$, implying that   $x = N(Q_n, C_{2\ell})\cdot 2\ell/ (n2^{n-1}) \leq   c_\ell n^{\ell} 2^n 2\ell/ (n2^{n-1})= c'_\ell n^{\ell-1}.$ 
Thus the maximum degree $D$ of the dependency graph satisfies  $D\leq 2\ell x \leq 2\ell \cdot c'_\ell n^{\ell-1}$.  \\

Observe that  $P(D+1)e  = (cn^{-a})^{2\ell-1}  (c'_\ell n^{\ell-1}+1) e$, it is less than $1$ if the constant $c$ in the definition of $p$ is sufficiently small and the exponent  $-a(2\ell -1)+{\ell -1} = 0$, i.e., if $a=(\ell-1)/(2\ell -1)$. Using Lov\'asz Local Lemma, 
we have that with a positive probability there are no bad events $A_C$. A largest color class gives a $C_{2\ell}$-free subgraph of $Q_n$ with at least $p||Q_n|| = cn^{1-a}2^n = 
cn^{\frac{\ell}{2\ell-1}} 2^n = cn^{\frac{1}{2}  + \frac{1}{4\ell -2}}2^n$ edges.
\end{proof}

\end{document}